\documentclass[letterpaper,11pt]{amsart}
\usepackage[utf8x]{inputenc}
\usepackage{amssymb, amsmath, amsthm}
\theoremstyle{definition}
\newtheorem{definition}{Definition}[section]

\newtheorem{theorem}{Theorem}[section]
\newtheorem{corollary}[theorem]{Corollary}
\newtheorem{lemma}[theorem]{Lemma}

\newcommand{\res}{\textsc{Res}}
\newcommand{\1}{\mathbf{1}}
\newcommand{\bu}{\mathbf{u}}
\newcommand{\bv}{\mathbf{v}}
\newcommand{\spec}{\text{spec}}
\newcommand{\cA}{\mathcal{A}}
\newcommand{\cD}{\mathcal{D}}
\newcommand{\cF}{\mathcal{F}}
\newcommand{\cI}{\mathcal{I}}
\newcommand{\bx}{\mathbf{x}}
\newcommand{\by}{\mathbf{y}}
\newcommand{\bz}{\mathbf{z}}
\newcommand{\Tr}{\text{Tr}}
\newcommand{\tr}{\text{tr}}
\newcommand{\supp}{\text{supp}}
\renewcommand{\thefootnote}{\fnsymbol{footnote}}

\allowdisplaybreaks

\title{Spectra of Uniform Hypergraphs}

\begin{document}

\maketitle

\begin{center} Joshua Cooper\footnote{\noindent cooper@math.sc.edu, This work was funded in part by NSF grant DMS-1001370.} and Aaron Dutle\footnote{\noindent dutle@mailbox.sc.edu, Corresponding author. } \\ Department of Mathematics\\ University of South Carolina\let\thefootnote\relax \let\thefootnote\relax
\end{center}
\begin{abstract}
We present a spectral theory of uniform hypergraphs that closely parallels
Spectral Graph Theory. A number of recent developments building upon classical work has led to a rich understanding of
``symmetric hyperdeterminants'' of hypermatrices, a.k.a.~multidimensional arrays.
Symmetric hyperdeterminants share many properties with determinants, but the
context of multilinear algebra is substantially more complicated than
the linear algebra required to address Spectral Graph Theory (i.e., ordinary matrices). Nonetheless, it is possible to define eigenvalues of
a hypermatrix via its characteristic polynomial as well as variationally. We
apply this notion to the ``adjacency hypermatrix'' of a uniform hypergraph, and prove a number of natural analogues of basic results in Spectral Graph Theory. Open problems abound, and we present a number of directions for further study.
\end{abstract}



\tableofcontents

\section{Introduction}

Spectral Graph Theory is a widely studied and highly applicable subject in combinatorics, computer science, and the social sciences.  Broadly speaking, one first encodes the structure of a graph in a matrix $M$ and then pursues connections between graph properties and the eigenvalues or singular values of $M$.  Work has addressed the ``adjacency matrix'', as well as other matrices that come from a graph: the ``(combinatorial) Laplacian'', the ``unsigned Laplacian'', the ``normalized Laplacian'', the ``bipartite adjacency matrix'', the ``incidence matrix'', and others (q.v.~\cite{big93,chu97,cve88,cve80}).  One natural avenue of study is the generalization of spectral techniques to hypergraphs, though there is a conspicuous paucity of results known in this vein.  There have been attempts in the literature to define eigenvalues for hypergraphs and study their properties, with varying amounts of success.  Notable examples include \cite{chu93,fl96,fw95,lu09,r09}.  Most of this work concerns generalizations of the {\it Laplacian} spectrum of a graph, and has a very different flavor than the subject we discuss in the sequel.

Unfortunately, na\"{i}vely attempting to generalize the spectral theory of adjacency matrices to hypergraphs runs into serious obstacles right away.  There is no obvious way to define an adjacency matrix for a hypergraph, since edges are specified by more than two vertices.  For a $k$-uniform hypergraph, one can obtain a straightforward generalization to an order-$k$ array (essentially, a tensor), but in doing so, one loses the powerful and sophisticated tools of linear algebra with which to analyze it.  Another tack, which is related to the $k$-dimensional array strategy, is to consider eigenvalues as the simultaneous vanishing of a system of equations each containing the parameter $\lambda$, where the eigenvalues are those $\lambda$ for which the system has a solution. This generalizes the idea that matrix eigenvalues are the vanishing of linear equations with $\lambda$ as a parameter.

Recent work \cite{cpz08,fgh10,lim05,russians,qi05} has provided some of the framework and tools with which to analyze such higher dimensional arrays, and we employ these developments extensively to define and analyze our hypergraph eigenvalues.  We obtain a number of results closely paralleling results from classical spectral graph theory, including bounds on the largest eigenvalue, a spectral bound on the chromatic number, and a sub-hypergraph counting description of the coefficients of the characteristic polynomial.  We also describe the spectrum for some natural hypergraph classes and operations, including disjoint unions, Cartesian products, $k$-partite graphs, $k$-cylinders, a generalization of the hypercube, and complete hypergraphs.

Recent work has used variations of the hypergraph eigenvalues we describe to obtain results about the maximal cliques in a hypergraph \cite{bp09a}, cliques in a graph based on hypergraphs arising from the graph \cite{bp09b}, and connectivity properties for hypergraphs of even uniformity \cite{hu11+}. We add many basic results about hypergraph eigenvalues to this body of work. Many of our results employ an algebraic characterization of the eigenvalues which has been underutilized in the study of hypergraphs.  Furthermore, our characterizations of the eigenvalues (and, in some cases, eigenvectors) of certain hypergraph classes gives several classes of hypermatrices for which we now understand the spectrum.  We hope that this work can provide a foundation for further study of the eigenvalues of symmetric hypermatrices.

The remainder of the paper is organized as follows. In Section 2, we give definitions and background on eigenvalues of symmetric hypermatrices, including both variational and algebraic formulations. In Section 3, we define the adjacency hypermatrix for a $k$-uniform hypergraph, and derive hypergraph generalizations of many of the central results of Spectral Graph Theory. Section 4 explores the spectra of several ``common'' hypergraphs: complete graphs, Cartesian products, $k$-cylinders, etc. Section 5 outlines a surfeit of directions for further study.

\section{Eigenvalues of Symmetric Hypermatrices}
We begin by defining the array that we will use to encode a $k$-uniform hypergraph.  We denote the set $\{1,\ldots,n\}$ by $[n]$.

\begin{definition} \label{hypermat} A (cubical) \emph{hypermatrix} $\cA$ over a set $\mathbb{S}$ of dimension $n$ and order $k$ is a collection of $n^k$ elements $a_{i_1 i_2 \ldots i_k} \in \mathbb{S}$ where $i_j \in [n]$.
\end{definition}

For the remainder of the present discussion, $\mathbb{S} = \mathbb{C}$.

\begin{definition} A hypermatrix is said to be \emph{symmetric} if entries which use the same index sets are the same.  That is, $\cA$ is symmetric if    $a_{i_1 i_2 \ldots i_k} = a_{i_{\sigma(1)} i_{\sigma(2)} \ldots i_{\sigma(k)}}$ for all $\sigma \in \mathfrak{S}_k$, where $\mathfrak{S}_k$ is the symmetric group on $[k]$.
\end{definition}

In the case of graphs, i.e., $k=2$, cubical hypermatrices are simply square matrices, and symmetric hypermatrices are just symmetric matrices.  It should be noted that some authors use the term \emph{tensor} in place of hypermatrix.  Strictly speaking, however, a hypermatrix is not simply a tensor: it is a tensor expressed in a particular basis.  It is worth noting that there are several additional departures from the above nomenclature in the literature.  Some authors (e.g., \cite{gkz94}) simply refer to hypermatrices as ``matrices'', while computer scientists often call them ``multidimensional arrays.''  Also, in order to use ``symmetric'' to refer to weaker notions of symmetric hypermatrices, some authors use ``supersymmetric'' to refer to what we term ``symmetric.''

An order $k$ dimension $n$ symmetric hypermatrix $\cA$ uniquely defines a homogeneous degree $k$ polynomial in $n$ variables (a.k.a.~a ``$k$-form'') by
\begin{equation} \label{form} F_{\cA}(\textbf{x})  = \sum_{i_1, i_2, \ldots , i_k = 1}^n  a_{i_1 i_2 \ldots i_k}x_{i_1}x_{i_2} \ldots x_{i_k}. \end{equation}

 Qi (\cite{qi05}) and Lim (\cite{lim05}) offered several generalizations of the eigenvalues of a symmetric matrix to the case of higher order symmetric (or even non-symmetric) hypermatrices.  We employ the following definition from \cite{qi05}.

\begin{definition} \label{eigdef} Call $\lambda\in \mathbb{C}$ an \emph{eigenvalue} of $\cA$ if there is a non-zero vector $\mathbf{x} \in \mathbb{C}^n,$ which we call an \emph{eigenvector}, satisfying

\begin{equation} \label{eigeq}  \sum_{i_2, i_3, \ldots, i_k=1}^n a_{ji_2i_3\ldots i_k}x_{i_2}\ldots x_{i_k} = \lambda x_j^{k-1}\end{equation}
for all $ j \in [n].$ \end{definition}

If we write $\mathbf{x}^r$ for the order $r$ dimension $n$ hypermatrix with $i_1, i_2, \ldots, i_r$ entry $x_{i_1}x_{ i_2} \ldots x_{i_r},$ and $\mathbf{x}^{[r]}$ for the vector with $i$-th entry $x_i^r$, then the expressions above can be written rather succinctly. Equation (\ref{form}) becomes
$$
F_{\cA}(\mathbf{x}) = \cA\mathbf{x}^k,
$$
where multiplication is taken to be tensor contraction over all indices. Similarly, the eigenvalue equations (\ref{eigeq}) can be written as
$$
\cA\mathbf{x}^{k-1} = \lambda\mathbf{x}^{[k-1]},
$$
where contraction is taken over all but the first index of $\cA$.  We mostly avoid the tensorial nomenclature, and work instead with polynomials, although we sometimes use the above notation for concision.

With the definitions above (and some generalizations of them), much work has been done using variational or analytic techniques, including conditions for the function $F_{\cA}(\mathbf{x})$ to be positive definite and Perron-Frobenius-type theorems (\cite{cpz08,fgh10,lim05}).
We rely on these results, but draw more from the algebraic approach suggested by Qi in \cite{qi05}, which uses a construction from algebraic geometry called the \emph{resultant}. We give a brief background and some useful properties of this construction.

\subsection{The Multipolynomial Resultant}

The resultant of two polynomials in one variable (or alternatively two homogeneous polynomials in two variables) is a classical construction used to determine if the two polynomials have a common root. It can be defined and calculated in a number of ways, including the determinant of the so-called ``Sylvester matrix'' of the two polynomials.  On the other extreme, if we have $n$ linear forms in $n$ variables, the determinant of the coefficient matrix tells us when these forms have a common non-trivial zero.  The multipolynomial resultant is a construction that unifies both concepts under a single framework. Readers wishing to learn more about the topic may find a highly algebraic treatment of the resultant and its generalizations in the text by Gelfand et al, \cite{gkz94}; those looking for a less specialized and more algorithmic approach may consult the text by Cox et al, \cite{clo98}.

We reproduce two theorems giving some important facts about the resultant.  For proofs of these results, see \cite{gkz94}.  First is the existence and (once suitably normalized) uniqueness of the resultant.

\begin{theorem} \label{resexist} Fix degrees $d_1, d_2, \ldots, d_n$. For  $i\in [n]$, consider all monomials $\mathbf{x}^\alpha$ of total degree $d_i$ in $x_1, \ldots, x_n$.  For each such monomial, define a variable $u_{i,\alpha}$. Then there is a unique polynomial $\res \in \mathbb{Z}[\{u_{i, \alpha}\}]$ with the following three properties:
\begin{enumerate}
\item If $F_1, \ldots, F_n \in \mathbb{C}[x_1, \ldots, x_n]$ are homogeneous polynomials of degrees $d_1$, $\ldots$, $d_n$ respectively, then the polynomials have a non-trivial common root in $\mathbb{C}^n$ exactly when $\res(F_1, \ldots, F_n) = 0$. Here, $\res(F_1, \ldots, F_n)$ is interpreted to mean substituting the coefficient of $\mathbf{x}^\alpha$ in $F_i$ for the variable $u_{i, \alpha}$ in $\res$.
\item $\res(x_1^{d_1}, \ldots, x_n^{d_n}) = 1$.
\item $\res$ is irreducible, even in $\mathbb{C}[\{u_{i,\alpha}\}]$.
\end{enumerate}
\end{theorem}

Next, the resultant is homogeneous in each group of coefficients.
\begin{theorem} \label{reshom} Fix degrees $d_1, \ldots, d_n$. Then for $i \in [n]$,  $\res$ is homogeneous in the variables $\{u_{i, \alpha}\}$ with degree $d_1 d_2 \ldots d_{i-1} d_{i+1} \ldots d_n$.
\end{theorem}

Since the equations $\cA \bx^{k-1}$ are a collection of $n$ homogeneous polynomials in $n$ variables, we can use them as ``input'' into the resultant, which leads to the following.

\begin{definition} \label{hyperdet} The \emph{symmetric hyperdeterminant} of $\cA$, denoted $\det(\cA)$, is the resultant of the polynomials $\cA\mathbf{x}^{k-1}$.
\end{definition}

Let $\cI$, the identity hypermatrix, have entries  $\displaystyle \begin{cases} 1 & \text{if } i_1=i_2=\ldots = i_k \\ 0 & \text{otherwise.} \end{cases} $
\begin{definition} \label{charpoly} Let $\lambda$ be an indeterminate. The \emph{characteristic polynomial} $\phi_{\cA}(\lambda)$ of a hypermatrix $\cA$ is $\phi_{\cA}(\lambda) = \det(\lambda \cI-\cA)$. \end{definition}

Qi, in \cite{qi05}, determined many properties of the symmetric hyperdeterminant and the characteristic polynomial, the most crucial being that the roots of the characteristic polynomial are exactly the eigenvalues of $\cA$. From this, we take our generalization of the spectrum of a matrix.

\begin{definition} The \emph{spectrum} of a symmetric hypermatrix $\cA$, denoted $\spec(\cA)$, is the set (or multiset, depending on the context) of roots of $\phi_\cA(\lambda).$
\end{definition}

\section{General Hypergraph Spectra}

In the sequel, we employ standard definitions and notation from hypergraph theory; see, e.g., \cite{ber89}.  A {\it hypergraph} $H$ is a pair $(V,E)$, where $E \subseteq \mathcal{P}(V)$.  The elements of $V = V(H)$ are referred to as {\it vertices} and the elements of $E = E(H)$ are called {\it edges}. A hypergraph $H$ is said to be {\it $k$-uniform} for an integer $k \geq 2$ if, for all $e \in E(H)$, $|e| = k$. We will often use the term \emph{$k$-graph} in place of  \emph{$k$-uniform hypergraph}. Given two hypergraphs $H = (V,E)$ and $H^\prime = (V^\prime, E^\prime)$, if $V^\prime \subseteq V$ and $E^\prime \subseteq E$, then $H^\prime$ is said to be a {\it subgraph} of $H$.  A set of vertices $S \subset V(H)$ is said to {\it induce} the subgraph $H[S] = (S,E \cap \mathcal{P}(S))$.  A $k$-uniform {\it multihypergraph} $H$ is a pair $(V,E)$, where $E$ is a multiset of subsets of $V$ of cardinality $k$.  Given a hypergraph $H = (V,E)$ and a multihypergraph $H^\prime = (V^\prime,E^\prime)$, if $(V^\prime,E^{\prime \prime})$ is a subgraph of $H$, where $E^{\prime \prime}$ is the {\em set} of elements of $E^\prime$, then $H^\prime$ is said to be a {\em multi-subgraph} of $H$.

\begin{definition} \label{adjmat} For a $k$-graph $H$ on $n$ labeled vertices, the (normalized) \emph{adjacency hypermatrix}
 $\cA_H$ is the order $k$ dimension $n$ hypermatrix with entries
$$
a_{i_1, i_2, \ldots, i_k} = \frac{1}{(k-1)!}\begin{cases}  1& \text{ if } \{i_1, i_2, \ldots i_k\} \in E(H) \\
                                               0 & \text{ otherwise.} \end{cases}
$$
\end{definition}

When dealing with the spectrum of the adjacency hypermatrix of a hypergraph, we often suppress the hypermatrix, writing $\spec(H)$ for $\spec(\cA_H)$,  $F_H(\bx)$ for $F_{\cA_H}(\bx)$, etc.

With Definition \ref{adjmat}, the function $F_{H}(\bx)$  and the eigenvalue equations (\ref{eigeq}) take on a particularly nice form. For an edge $e=\{i_1, i_2, \ldots, i_r\}$ of an $r$-graph, let $x^e$ denote the monomial $x_{i_1}x_{i_2} \ldots x_{i_r}$. Recall that the \emph{link}  of a vertex $i$ in $H$, denoted $H(i)$, is the $(k-1)$-graph whose edges are obtained by removing vertex $i$ from each edge of $H$ containing $i$. That is, $E(H(i)) = \{ e \setminus \{i\} \, | \, i \in e \in E(H) \}$ and $V(H(i)) = \bigcup E(H(i))$.  Then
$$
F_{H}(\mathbf{x}) = \sum_{e \in H} kx^e,
$$
and the eigenvalue equations (\ref{eigeq}) become
\begin{equation} \label{eq:linkform}
\sum_{e\in H(i)} x^e = \lambda x_i^{k-1},
\end{equation}
for all $i \in V(H)$.  The normalization factor $\frac{1}{(k-1)!}$  in Definition \ref{adjmat} is included essentially for aesthetic reasons. It could easily be absorbed into $\lambda$ in the eigenvalue equations without altering any of the calculations. Normalization allows the adjacency hypermatrix to faithfully generalize the adjacency matrix of a graph while removing a factor of $(k-1)!$ that would otherwise make an appearance in some of the results below.

Using the definitions given Section 2, a number of results from basic spectral graph theory can be generalized to the $k$-graph case in a natural way. Indeed, some results need only slight modifications of their standard proofs.  Others -- particularly those that give results about multiplicities -- require the use of new techniques.

\begin{theorem} \label{disjoint}  Let $H$ be a $k$-graph that is the disjoint union of hypergraphs $H_1$ and $H_2$. Then as sets, $\spec(H) =\spec(H_1)\cup \spec(H_2)$.  Considered as multisets, an eigenvalue $\lambda$ with multiplicity $m$ in $\spec(H_1)$ contributes $\lambda$ to $\spec(H)$ with multiplicity $m(k-1)^{|H_2|}.$
\end{theorem}

We first prove a more general lemma about the resultant of a system of polynomials which can be viewed as the union of two disjoint systems. The proof of the theorem is then a simple application of this lemma.

\begin{lemma} \label{djres} Let $F_1, F_2, \ldots, F_{n} \in \mathbb{C}[x_1, \ldots, x_n]$ be  homogeneous polynomials of degrees $d_1,\ldots, d_n,$ and let $ G_1, G_2, \ldots, G_m \in \mathbb{C}[y_1, \ldots, y_m]$  be homogeneous polynomials of degrees $\delta_1, \ldots , \delta_m.$  Then
$$
\res(F_1, \ldots, F_n, G_1, \ldots G_m) = \res(F_1, \ldots, F_n)^{\prod_i \delta_i} \res(G_1, \ldots G_m)^{\prod_i d_i}.
$$
\end{lemma}

\begin{proof} Instead of considering particular polynomials, we work with ``generic'' polynomials. We consider each $F_i$
as having a distinct variable $a_{i,\alpha}$ as a coefficient for each monomial $x^\alpha$ of degree $d_i$
in the $x$ variables (and the corresponding coefficient variables $b_{i, \beta}$ for
$G_i, \delta_i,$ and the $y$ variables).

Then by Theorem \ref{resexist} we can consider $\res(F_1, \ldots, F_n)$ as an irreducible integer polynomial in
$\mathbb{C}\left[\{a_{i, \alpha}\}\right]$ (which is also irreducible in $\mathbb{C}\left[\{a_{i, \alpha}\}, \{b_{i, \beta}\}\right]$).
Similarly $\res(G_1, \ldots, G_m)$ is an irreducible polynomial in the $b_{i,\beta}$ variables.

Now consider $\res(F_1, \ldots, F_n)\res(G_1, \ldots, G_m).$  It is a polynomial in all of the coefficient variables, and if we consider
$\mathbb{C}^{n+m} = \mathbb{C}^n\times\mathbb{C}^m$, this polynomial takes value zero precisely when at least one of the systems $\{F_i\}, \{G_i\}$
has a non-trivial solution in its respective space. Thus, any zero of this polynomial is a setting of the coefficient variables $\{a_{i, \alpha}, b_{j, \beta}\}$ so that at least one of the two systems described by the coefficients has a non-trivial solution.

Similarly,  $\res(F_1, \ldots, F_n, G_1, \ldots, G_m)$ is an integer polynomial in all of the coefficient variables, which takes value zero when
the entire system has a non-trivial solution in $\mathbb{C}^{n+m}$. Consequently, this resultant taking value zero gives that at least one of the systems $\{F_i\}, \{G_i\}$ has a non-trivial solution.  Notice, however, that if the system $\{F_i\}$ has a
nontrivial solution, then setting all of the $y$ variables to zero gives a nontrivial solution to the entire system.
Hence $\res(F_1, \ldots, F_n, G_1, \ldots, G_m)$ takes value zero precisely for assignments of the coefficient variables where at least one of the systems $\{F_i\}, \{G_i\}$ described by these coefficients has a non-trivial solution in its respective space.

Since these two polynomials have exactly the same zeroes, they can only  differ (up to a unit) in the multiplicities of their
irreducible factors. Since each of $\res(F_1, \ldots, F_n),$ and $\res(G_1, \ldots, G_m)$ are already irreducible, we have that for some
integers $D, \Delta >0$, and complex number $c\neq 0$,
$$
\res(F_1, \ldots, F_n, G_1, \ldots, G_m) = c\res(F_1, \ldots, F_n)^\Delta\res(G_1, \ldots, G_m)^D.
$$

Theorem \ref{reshom} gives us that $\res(F_1, \ldots \, F_n)$ is homogeneous of
degree $d_2 \ldots d_n$ in the coefficients of $F_1$, while
$\res(F_1, \ldots, F_n, G_1, \ldots, G_m)$ is homogeneous of degree
$d_2 \ldots d_n \delta_1 \ldots \delta_m$ in the coefficients
of $F_1$.  Since $\res(G_1, \ldots, G_m)$ does not involve any of the $a_{i, \alpha}$,
we may conclude that $\Delta = \prod_j \delta_j$. Similarly, we have that $D = \prod_j d_j$.  Finally, property (b) of Theorem \ref{resexist} implies that $c=1$.
\end{proof}

\begin{proof}[Proof of Theorem \ref{disjoint}]  Let $H, H_1,$ and $H_2$ be as in the statement of the theorem, and let $\phi_H(\lambda)$, $\phi_{H_1}(\lambda)$, $\phi_{H_2}(\lambda)$ denote their respective characteristic polynomials. To prove the theorem, it suffices to show
$$
\phi_H(\lambda) = \phi_{H_1}(\lambda)^{(k-1)^{|H_2|}} \phi_{H_2}(\lambda)^{(k-1)^{|H_1|}}.
$$

Since $H$ is the disjoint union of two hypergraphs, the polynomials in (\ref{eq:linkform}) can be partitioned into two sets $\Sigma_1$ and $\Sigma_2$, where $\Sigma_i$ uses only variables corresponding to vertices of $H_i$, $i=1,2$.  Noting that the degree of each of these polynomials is $k-1$, and that the characteristic polynomial $\phi_H$ is the resultant of the entire system $\Sigma_1 \cup \Sigma_2$, Lemma \ref{djres} gives the desired result.
\end{proof}

\subsection{Properties of the Largest Eigenvalue}

We derive some properties of the eigenvalue of a $k$-graph $H$ with largest modulus, which we denote $\lambda_\text{max} = \lambda_\text{max}(H)$.  Although {\em a priori} there may be many such eigenvalues with the same modulus, we show below that there is always a real, positive one\footnote[1]{Unless the hypergraph has no edges, in which case all of the eigenvalues are trivially zero.}.
In order to proceed, we begin with a few facts about eigenvalues associated with positive eigenvectors.

Let $H$ be a $k$-graph with $n$ vertices, and let
$$
S_{\geq 0} = \{\bx\in \mathbb{R}^n \, | \, \sum_{i=1}^nx_i^k = 1 \text{ and } x_i\geq 0 \,\text{ for } \, i \in [n]  \}.
$$
For a vector $\bx$, we call the set $\supp(\bx)$ of all indices of non-zero coordinates of $\bx$ the \emph{support} of the vector.  For ease of notation, throughout the sequel we identify coordinate indices of vectors with the corresponding vertices of the hypergraph under consideration.

\begin{lemma} If $\bv \in S_{\geq 0}$ maximizes $F_H(\bx)$ on $S_{\geq 0}$, then $\supp(\bv)$ induces some collection of connected components of $H$. \end{lemma}

\begin{proof} Suppose, by way of contradiction, that $\bv$ maximizes $F_H(\bx)$, but the support of $\bv$ is not the vertex set of a collection of connected components of $H$.  In particular, $V(H) \neq \supp(\bv)$.  Define $I = V(H) \setminus \supp(\bv)$, and partition the edges of $H$ as follows.  Let $\cF_1$ be the set of edges of $H$ using only vertices of $\supp(\bv)$, let $\cF_2$ be the set of edges of $H$ using only the vertices of $I$, and let $\cF_3 = E(H) \setminus (\cF_1 \cup \cF_2)$. Since the support of $\bv$ does not induce a collection of components, we have $\cF_3 \neq \emptyset$. Let $\hat{e} \in \cF_3$ be one such edge.

Let $v_{\text{min}}$ be the smallest non-zero entry of $\bv$. For $0 < s <1,$ define $\delta = \left(\frac{1-s^k}{|I|}\right)^{1/k}.$ Note that $\delta>0, $ but tends to zero as $s$ tends to 1. Hence we can find $s_0$ with $1/2<s_0<1$ so that $\delta \leq v_{\text{min}}/2$ for any $s$ with $s_0<s <1$. Thus for these $s$, we have \begin{equation} \label{coordbound} \delta \leq v_{\text{min}}/2 \leq s v_i\end{equation}  for any non-zero entry $v_i$ of $\bv.$

Next, let  $\bz\in \mathbb{R}^n$ be the vector whose support is $I$, and whose non-zero entries are each $1$.  Define the vector $\by = s \bv + \delta \bz.$
A quick verification shows that
\begin{align*} \sum_{i=1}^n y_i^k & = \sum_{i\in [n]\setminus I} y_i^k + \sum_{i\in I}y_i^k \\
& = \sum_{i\in [n]\setminus I} (s v_i)^k + \sum_{i\in I}\delta^k \\
& = s^k \sum_{i=1}^n v_i^k + \sum_{i\in I} \frac{1-s^k}{|I|} \\
& = s^k +(1-s^k) = 1,
\end{align*}
so that  $\by \in S_{\geq 0}$. Note that the support of $\by$ is all of $H$.

As $\bv$ maximizes $F_H(\bx)$, we have that

\begin{align*}
F_H(\bv)  & \geq F_H(\by) \\
& = k\sum_{e\in H}y^e \\
& = k\left( \sum_{e\in \cF_1} y^e + \sum_{e\in \cF_2}y^e + \sum_{e\in \cF_3}y^e\right) \\
& \geq k\left( \sum_{e\in \cF_1}s^k v^e\right) + ky^{\hat{e}} \\
& = s^kF_H(\bv) +ky^{\hat{e}} \end{align*}

Solving this inequality for $F_H(\bv),$ we see

\begin{equation} \label{topeigcont} F_H(\bv)\geq k (1-s^k)^{-1} y^{\hat{e}} \end{equation}

If we write $\hat{e} = \{i_1, i_2, \ldots, i_k\}$, then  $y^{\hat{e}} = y_{i_1}y_{i_2}\ldots y_{i_k}.$ By the definition of $\by$, we have $$y_{i_j} = \begin{cases} \delta & \text{ if } i_j\in I \\ s v_{i_j} & \text{ otherwise. } \end{cases}$$ We know that $\hat{e}$ uses vertices of $I$ and $V(H)\setminus I$. In particular, it has at least one factor $s v_{i_j}.$ Using (\ref{coordbound}), we may write $s v_{i_j} \geq v_{\text{min}}/2$ and bound all others factors from below by $\delta$. Thus we have  $$y^{\hat{e}} \geq \delta^{k-1} \frac{v_{\text{min}}}{2}.$$ Substituting this into (\ref{topeigcont}), we see
\begin{align*}
F_H(\bv) & \geq \frac{kv_{\text{min}}}{2(1-s^k)} \delta^{k-1} \\
& = \frac{kv_{\text{min}}}{2(1-s^k)} \left( \frac{1-s^k}{|I|}\right)^{\frac{(k-1)}{k}} \\
& =C (1-s^k)^{-1/k}
\end{align*}
where $C = \frac{kv_{\text{min}}}{2|I|^{(k-1)/k}}$ is a constant.  Note that the left side is a fixed value, while the right side becomes arbitrarily large as $s$ tends to 1. Therefore, we have a contradiction. \end{proof}

From this lemma we can derive the following useful corollary.

\begin{corollary} \label{realpair} If $H$ is a connected $k$-graph, then it has a strictly positive eigenpair $(\lambda, \bv)$ where $\lambda$  is the maximum value of $F_H(\bx)$ on $S_{\geq 0}.$   \end{corollary}

\begin{proof} Since $H$ is connected, the lemma tells us that the maximum of $F_H(\bx)$ is obtained by a vector $\bv$ with full support. So  $F_H(\bx)$ is maximized on the interior of $S_{\geq 0}.$ Since this set is compact, and $F_H(\bx)$ achieves its maximum in the interior, $\bv$ must be a critical point of $F_H(\bx)$.  The critical points of $F_H(\bx)$ are exactly the eigenvectors of $H$, so the maximum happens at an eigenvector. However, on $S_{\geq 0}$, if $(\lambda, \bv)$ is an eigenpair, it is easy to see that $F_H(\bv) = \lambda$.  Furthermore, since all entries of $\bv$ are positive, $\lambda = F_H(\bv)>0$.    \end{proof}

Once we have the existence of a strictly positive eigenpair, the methods used in \cite{cpz08} allow us show that the eigenvalue from Corollary \ref{realpair} is actually the largest eigenvalue.  We include an adapted proof for completeness.

\begin{lemma} \label{largepos} If $(\lambda, \bv)$ is a strictly positive eigenpair, and $(\mu, \by)$ are non-negative (i.e., $\by\neq 0$) with the property that $\sum_{e\in H(i)}y^{e}\geq \mu y_i^{k-1}$ for each $i$,  then $\mu \leq \lambda$. \end{lemma}

\begin{proof} Since $\bv$ is strictly positive, we can find a $t_0$ so that (coordinatewise), $\bv\geq t\by$ for all $0\leq t \leq t_0$, but at least one of the inequalities fails for any $t>t_0$. From the inequality $\bv\geq t_0\by,$ it is clear that for each $i$, we have  $\sum_{e\in H(i)}v^e \geq \sum_{e\in H(i)}(t_0y)^{e}.$
Thus,
\begin{align*}
\lambda v_i^{k-1} & = \sum_{e\in H(i)}v^e \\
& \geq \sum_{e\in H(i)}(t_0y)^{e} \\
& = t_0^{k-1} \sum_{e\in H(i)}y^{e}  \\
& \geq t_0^{k-1} \mu y_i^{k-1}.
\end{align*}
Solving for $v_i$, we find
$$
v_i \geq \left(\frac{\mu}{\lambda}\right)^{1/(k-1)} t_0 y_i .
$$
As this holds for all $i$, we see that $\bv \geq \left(\frac{\mu}{\lambda}\right)^{1/(k-1)} t_0 \by.$  Since $t_0$ was chosen to be the largest value that makes this inequality hold, we may conclude that $\left(\frac{\mu}{\lambda}\right)^{1/(k-1)} \leq 1$, which gives $\mu \leq \lambda$.
\end{proof}

This simple lemma allows us to deduce a few important properties.

\begin{corollary} \label{reallargest} If $H$ is a connected $k$-graph, then the real eigenvalue $\lambda$ given by Corollary \ref{realpair} is the only eigenvalue with a strictly positive eigenvector.  If $\nu$ is any other eigenvalue of $H$, then $|\nu|\leq \lambda$.
\end{corollary}

\begin{proof} For the first statement, suppose that  $(\lambda, \bv)$ and $(\mu, \by)$ are both strictly positive eigenpairs. Applying Lemma \ref{largepos} gives that $\mu\leq \lambda$.  Switching the roles of $\mu$ and $\lambda$ and applying the same lemma gives that $\lambda \leq \mu$.  Thus, the first statement holds.

For the second, suppose that $(\nu, \bz)$ is any eigenpair. If we set $\mu = |\nu|$ and define $\by = (|z_1|, |z_2|, \ldots, |z_n|)$, we see that for each $i,$
\begin{align*}
\mu y_i^{k-1} & =  |\nu||z_i|^{k-1} \\
& = \left|\nu z_i^{k-1}\right| \\
& = \left|\sum_{e\in H(i)} z^e\right| \\
& \leq \sum_{e\in H(i)} |z|^e \\
& = \sum_{e\in H(i)} y^e
\end{align*}
 Lemma \ref{largepos} then applies to the pair $(\mu, \by)$ to show that $|\nu|\leq \lambda$.
\end{proof}

The next theorem summarizes the results on $\lambda_{\text{max}}.$
\begin{theorem} \label{maxeigenreal} For any non-empty $k$-graph $H$, $\lambda_\text{max}$ can be chosen to be a positive real number. If $H$ is connected, then a corresponding eigenvector $\bx$ can be chosen to be strictly positive.
\end{theorem}
\begin{proof} By Theorem \ref{disjoint}, we see that $\lambda_{\text{max}}$ is obtained as an eigenvalue of some connected component of $H$. For each connected component of $H$, Corollary \ref{reallargest} gives that the largest eigenvalue is real. The second statement is the result of Corollaries \ref{realpair} and \ref{reallargest}. \end{proof}

The following theorem is an analogue of a classical theorem in spectral graph theory, relating $\lambda_{\text{max}}$ to the average degree and maximum degree of the hypergraph. It also follows quickly from the above results.

\begin{theorem}  \label{degreebound} Let $H$ be a $k$-graph. Let $d$ be the average degree
of $H$, and $\Delta$ be the maximum degree. Then
$$
d \leq \lambda_{\text{max}} \leq \Delta.
$$
\end{theorem}
In particular, if we have a regular $k$-graph, $d=\lambda_\text{max} = \Delta.$

\begin{proof}
We note that Theorem \ref{disjoint} allows us to assume that $H$ is connected. By Corollaries \ref{realpair} and \ref{reallargest}, and Theorem \ref{maxeigenreal}, 
 we have that $\lambda_\text{max} \geq F_H(\bx)$ for any vector $\bx \in S_{\geq 0}$. If we let $\1$ be the vector with all entries equal to $1/ \sqrt[k]{n},$ we see that $\lambda_{\text{max}}\geq F_H(\1) = d.$

For the upper bound, let $\hat{\bv}$ be a vector achieving $\lambda_{\text{max}}$. Let $\hat{v}_i$ be the entry of $\hat{\bv}$ with largest modulus.
Rescale $\hat{\bv}$ to a vector $\bv$ where $v_i = 1.$

Then the $i^{\text{th}}$ eigenvalue equation gives that
$$ \lambda_{\text{max}}v_i^{k-1} = \sum_{e\in H(i)} v^e.$$

Hence \begin{align*} \lambda_{\text{max}} & = |\lambda_{\text{max}}v_i^{k-1}| \\
& = \left| \sum_{e\in H(i)} v^e \right| \\
& \leq \sum_{e\in H(i)}  |v^e| \\
& \leq \sum_{e\in H(i)} |v_i^{k-1}| \\
& = \sum_{e\in H(i)} 1  \\
&  = \deg(i) \leq \Delta .
\end{align*}
\end{proof}

\begin{theorem} \label{subgraph} If $G$ is a subgraph of $H$, then
$$
\lambda_\text{max}(G)\leq \lambda_\text{max}(H).
$$
\end{theorem}
\begin{proof} By Theorem \ref{disjoint}, we can assume that $G$ and $H$ are both connected. Let $F_G(\bx)$ and $F_H(\bx)$ be their associated homogeneous forms. We note that both have coefficients in $\{0, k\}$, and every term in $F_G(\bx)$ appears in $F_H(\bx)$. Let $\bv$ be the vector from the set ${S}_{\geq 0}$ that achieves $\lambda_\text{max}(G)$. Let $\bu$ be the same vector, with zero entries for any vertices that $H$ has, but $G$ lacks. Then we have
$$
\lambda_\text{max}(G)  = F_G(\bv) \leq F_H(\bu).
$$
By Theorem \ref{maxeigenreal}, we have that $F_H(\bu)\leq \lambda_\text{max}(H)$.
\end{proof}

\subsection{Chromatic Number and the Largest Eigenvalue}
For a hypergraph $H$, a function $f:V(H)\rightarrow [r]$ is a (weak) proper $r$-coloring of $H$ if for every edge $e = \{v_1, v_2, \ldots v_k\},$ there exist $i\neq j$ so that $f(v_i)\neq f(v_j).$ Informally, no edge has all of its vertices colored the same. The (weak) \emph{chromatic number} of $H,$ denoted $\chi(H)$, is the minimum $r$ such that $H$ has a proper $r$-coloring.

\begin{theorem} \label{chromatic} For any $k$-graph, $\chi(H)\leq \lambda_{\text{max}}(H)+1.$
\end{theorem}

Our proof is a reprise of the classical proof by Wilf in \cite{wil67}, which uses the coloring method described by Brooks \cite{bro41}.

\begin{proof} Define an ordering on the vertices of $H$ as follows. Let $H(n)=H$, and let $v_n$ be a vertex of smallest degree in $H(n)$. Inductively, we let $H(m)$ be the subgraph that remains after deleting vertex $v_{m+1}$ from $H(m+1),$ and let $v_m$ be a vertex of smallest degree in $H(m)$. We use the ordering $v_1, v_2, \ldots, v_n$ as input to a greedy coloring algorithm, which assigns to vertex $v_i$ the smallest  natural number so that $H(i)$ is still properly colored.

Now, note that $\chi(H(1)) = 1 \leq \lambda_{\text{max}}(H)+1$, as we know $\lambda_{\text{max}}\geq 0$. Inductively, we assume that we have properly colored $H(m)$ with at most $\lambda_{\text{max}}(H)+1$ colors. Note that vertex $v_{m+1}$ has the smallest degree of all vertices in $H(m+1)$. In the worst case, each edge containing $v_{m+1}$ are, aside from $v_{m+1}$, monochrome, and use the colors $\{1,2,\ldots, \deg_{H(m+1)}(v_{m+1})\}.$ In this situation, we need to use color $\deg_{H(m+1)}(v_{m+1})+1$ for $v_{m+1}$. Hence,    we see that to color vertex $v_{m+1}$, we need at most
\begin{align*}
\deg_{H(m+1)}(v_{m+1}) +1 & = \delta(H(m+1))+1 \\
&  \leq d(H(m+1))+1 \\
& \leq \lambda_{\text{max}}(H(m+1))+1 \\
& \leq \lambda_{\text{max}}(H)+1.
\end{align*}
The first inequality above is trivial, as the average degree is always at least the minimum degree, and the last two inequalities follow from Theorems \ref{degreebound}  and \ref{subgraph} respectively. \end{proof}

By Theorem \ref{degreebound}, we recover the $k$-graph analogue of Brooks' bound on the chromatic number.
\begin{corollary} \label{thm:brooks} For any $k$-graph $H$,  $\chi(H)\leq \Delta(H) +1$. \end{corollary}

Also implicit in our proof of Theorem \ref{chromatic} is that $\Delta(H)$ in Theorem \ref{thm:brooks} can also be replaced by the \emph{degeneracy} of the hypergraph, which is $\max_{G \subseteq H} \delta(G)$.

If we consider $\Delta(H)$ tending to infinity and $k$ fixed, Corollary \ref{thm:brooks} can be improved further using probabilistic methods. Indeed, a simple application of the Lov\'{a}sz Local Lemma \cite{as08} gives that
$$
\chi(H) \leq \left [ e(k \Delta(H) + 1) \right ]^{1/(k-1)} = O\left(\Delta(H)^{1/(k-1)}\right).
$$
However, if $\Delta(H) \gg d(H)$, it is possible that $\lambda_{\textrm{max}}$ is substantially smaller than $\Delta(H)$, so that Theorem \ref{chromatic} is still a better bound in some cases.

\subsection{Coefficients of the Characteristic Polynomial}

The characteristic polynomial is defined as the resultant of a certain system of equations, so calculating the characteristic polynomial requires computation of the resultant.  In \cite{russians}, Morozov and Shakirov  give a formula (using somewhat different notation) for calculating $\det(\cI-\cA)$ using ``Schur polynomials'' in the \emph{generalized traces} of the order $k$, dimension $n$ hypermatrix $\cA$.

\begin{definition} Define the $d$-th Schur polynomial $P_d \in \mathbb{Z}[t_1,\ldots,t_d]$ by $P_0 = 1$ and, for $d > 0$,
$$
P_d(t_1,\ldots,t_d) = \sum_{m=1}^d \sum_{\substack{d_1 + \cdots + d_m = d \\ \forall i (d_i > 0)}} \frac{t_{d_1} \cdots t_{d_m}}{m!}.
$$
More compactly, one may define the $P_d$ by writing
$$
\exp \left ( \sum_{d=1}^\infty t_d z^d \right ) = \sum_{d=1}^\infty P_d(t_1,\ldots,t_d) z^d.
$$
\end{definition}

Let $f_i$ denote the $i$-th coordinate of $\cA\bx^{k-1}$.  Define $A$ to be an auxiliary $n \times n$ matrix with distinct variables $A_{ij}$ as entries.  For each $i$, we define the differential operator
$$
\hat{f}_i = f_i\left(\frac{\partial}{\partial A_{i1}},\frac{\partial}{\partial A_{i2}}, \ldots, \frac{\partial}{\partial A_{in}}\right)
$$
in the natural way.  (To be precise, let $\mathcal{O}$ be the operator algebra over $\mathbb{C}$ generated by the differential operators $\{\partial/\partial A_{ij}\}_{i,j=1}^n$. Then $\hat{f}_i$ is the image of $f_i(x_1,\ldots,x_n)$ under the homomorphism from $\mathbb{C}[x_1,\ldots,x_n]$ to $\mathcal{O}$ defined by $x_j \mapsto \partial/\partial A_{ij}$ for each $j \in [n]$.)  For $d>0$, define the generalized $d$-th trace $\Tr_d(\cA)$ by
$$
\Tr_d(\cA) = (k-1)^{n-1} \sum_{d_1+d_2+\ldots+d_n=d} \left(\prod_{i=1}^n\frac{\hat{f}_i^{d_i}}{(d_i(k-1))!}\right)\tr(A^{d(k-1)}),
$$
where $\tr(\cdot)$ denotes ordinary matrix trace.  The authors prove that
$$
\det(\cI -\cA)= \exp\left(\sum_{i=1}^\infty -\frac{\Tr_i(\cA)}{i}\right).
$$
The right hand side of this equation can be expanded as a power series using the Schur polynomials, where we see
$$
\det(\cI-\cA) = \sum_{i=0}^\infty P_i\left(-\frac{\Tr_1(\cA)}{1}, \ldots, -\frac{\Tr_i(\cA)}{i}\right)
$$
Morozov and Shakirov note that since we know the degree of the resultant, and since the $d$-th Schur polynomial in the above expression is homogeneous of degree $d$ in the coefficients of $\cA$,  the resultant of $\cA$ (up to sign) is simply the Schur polynomial of the correct degree in this expression.

We note that this same reasoning tells us that the codegree $d$ coefficient of the characteristic polynomial of $\cA$ is  the degree $d$ part of the above expression, which is exactly the $d$-th Schur polynomial in the expression above.

This gives us a concrete way of finding the coefficients of the characteristic polynomial for any particular hypermatrix. It also gives us a tool to analyze the symmetric hypermatrix with variable entries $a_{i_1 i_2 \ldots i_k}$. Our first application of the technique is to prove that monomials in the symmetric hyperdeterminant have a very particular form.

 \begin{definition}\label{kvalentdef} Let $R$ be any ring, and let
$$
V = \{a_{i_1 i_2 \ldots i_k} \, | \, i_j\in [n],\,  j \in [k], a_{i_1 i_2\ldots i_k} = a_{i_{\sigma(1)} i_{\sigma(2)}\ldots i_{\sigma(k)}} \forall \sigma \in \mathfrak{S}_k\}
$$
be a set of variables.  A monomial $M$ in $R[V]$  is called \emph{$t$-valent} ($t\geq 2$) if every index $i$ appearing in a subscript of some variable occurring in $M$ appears $ 0 \pmod t$ times in $M$.
\end{definition}
For example, $a_{111}a_{001}$ is 2-valent (or ``bivalent''), $a_{011}a_{100}$ is 3-valent (``trivalent''), while $a_{011}a_{222}$ has no valency.

\begin{theorem} \label{kvalent} If $A$ is the order $k$ dimension $n$ symmetric hypermatrix with variable entries, every term of every coefficient of $\phi_A(\lambda)$ is $k$-valent.\end{theorem}

\begin{proof}Note that since terms arise from multiplying generalized traces together, it suffices to show that each generalized trace produces only $k$-valent terms. Since the traces are sums of terms of the form
$$
\left(\prod_{i=1}^n\frac{\hat{f}_i^{d_i}}{(d_i(k-1))!}\right)\tr(A^{d(k-1)}),
$$
for some $d_1+\ldots+d_n = d$, it suffices to show that each of these terms produces only $k$-valent terms.

We say that a term of $\tr(A^{d(k-1)})$ \emph{survives} an operator if it is nonvanishing under the action of the operator. For a single term of an operator involving $r$ differentiations and a single term of $\tr(A^r)$, this amounts to the variables of the trace term being in one-to-one correspondence with the differentiation variables.

If we ignore scalar factors, any single term of $\cD = \prod_{i=1}^n \hat{f}_i^{d_i}$ consists of the product of $d$ differential operators, $d_i$ of them of the form
$$
a_{i j_2 \ldots j_k}\frac{\partial}{\partial A_{i,j_2}}\frac{\partial}{\partial A_{i,j_3}} \ldots \frac{\partial}{\partial A_{i,j_k}}
$$
for each $i$.  When a variable $a_{i j_2 \ldots j_k}$ is included in such a product arising from the operator $\cD_i = \hat{f}_i^{d_i}$, call $i$ the variable's \emph{primary} index, and call the other indices \emph{secondary}.  Clearly, $i$ appears $d_i$ times as a primary index in any term of $\cD_i$.  To show that the only monomials that survive $\cD$ are $k$-valent, we show that a term of $\tr(A^{d(k-1)})$ that survives must use $i$ exactly $d_i(k-1)$ times among the secondary indices.

Note that if $i$ is a primary index for a variable, then each of the $k-1$ partial derivatives that accompany it have $i$ as the first index of the differentiation variable.  Recall that
$$
\tr(A^r) = \sum_{i_1, i_2, \ldots i_r} A_{i_1,i_2}A_{i_2,i_3}\ldots A_{i_{r-1},i_r}A_{i_r,i_1}.
$$
For such a monomial appearing in $\tr(A^{d(k-1)})$ to survive, for each $i \in [n]$, $i$ must occur as a first index of some $A_{st}$ variable $d_i(k-1)$ times. Hence, by the form of the trace monomials, it also occurs as a second index the same number of times.  However, the second indices of the differential operator variables correspond exactly to the secondary indices of the variables $a_{j j_2 \ldots j_k}$ from the terms of $\cD$. Therefore $i$ appears exactly $d_i(k-1)$ times as a secondary index as well, completing the proof.
\end{proof}

Using the methods above, we can fully describe the first few coefficients of the characteristic polynomial for a general $k$-graph.

\begin{theorem} \label{zerocoeff} For a $k$-graph $H$, the codegree $1,2,\ldots, k-1$ coefficients of $\phi_H(\lambda)$ are zero.\end{theorem}

\begin{proof} We consider an adjacency hypermatrix filled with variables whose indices label the possible edges of a $k$-graph. Then any monomial in these variables can be thought of as a multi-subgraph of $H$. By the previous theorem, we see that the codegree $d$ coefficient is made up of constant multiples of $k$-valent monomials of degree $d$, which correspond to multi-subgraphs; we extend the definition of $k$-valency to multihypergraphs in this way.  Noting that each variable in a hypergraph monomial uses $k$ distinct indices, we see that there are no $k$-valent subgraphs with fewer than $k$ edges, proving our claim.
\end{proof}

\begin{corollary} \label{zerotraces} For a $k$-graph,   $\Tr_i(H) = 0$ for $1\leq i < k$.
\end{corollary}
\begin{proof}
The proof of Theorem \ref{kvalent} tells us that generalized traces only produce $k$-valent terms, and the proof of Theorem \ref{zerocoeff} gives that there are no $k$-valent subgraphs on fewer than $k$ edges. \end{proof}  

With a bit more work, we can also characterize the codegree $k$ coefficient in terms of the number of edges of the $k$-graph.

\begin{theorem} \label{codegreek} For a $k$-graph $H$, the codegree $k$ coefficient of $\phi_H(\lambda)$ is $-k^{k-2}(k-1)^{n-k} |E(H)|$.  \end{theorem}

Note that setting $k=2$ recovers the well-known fact that the codegree 2 coefficient of a graph's characteristic polynomial counts the number of edges; for higher uniformity, one obtains a multiple of the number of edges which depends on the number of vertices as well.

\begin{proof} First, recall that the codegree $k$ coefficient is
$$
P_k\left(-\frac{\Tr_1(H)}{1}, \ldots, -\frac{\Tr_k(H)}{k} \right ).
$$
By Corollary \ref{zerotraces}, all but the last parameter of this function are zero, so that the desired coefficient is $-\Tr_k(H)/k$ (by the decomposition of $P_k$ appearing in the proof of Corollary \ref{zerotraces}).  Recalling the definition of the trace,
\begin{align*}
\frac{\Tr_k(H)}{k} & = \frac{(k-1)^{n-1}}{k}\sum_{d_1+d_2+\ldots+d_n=k} \left(\prod_{i=1}^n\frac{\hat{f}_i^{d_i}}{(d_i(k-1))!}\right)\tr(A^{k(k-1)}). \end{align*}

We now show that most of the terms in sum in the above expression make no contribution to the trace. Indeed, the only operators that have surviving terms are those with each $d_i\leq 1,$ and where the indices $i$ with $d_i=1$ describe an edge of $H$.  To see this, note that the only $k$-valent subgraph on $k$ edges is a single edge repeated $k$ times. From the proof of Theorem \ref{kvalent}, we see that the only indices appearing in surviving terms must occur as a primary index in some term of $\prod_{i=1}^n \hat{f}_i^{d_i}$, and that the associated coefficient is non-zero only when these indices actually describe an edge.

This lets us further refine our expression for the codegree $k$ coefficient to include only those operators whose coefficients correspond to a {\em bona fide} edge. That is,
$$
\frac{\Tr_k(H)}{k} = \frac{(k-1)^{n-1}}{k\left((k-1)!\right)^k} \sum_{e\in H}\left(\prod_{i\in e}\hat{f}_i\right) \tr(A^{k(k-1)}).
$$
In fact, we can say substantially more. For a fixed edge $e\in H$, there is only one term of the operator $\prod_{i\in e}\hat{f}_i$ that has survivors: the one that arises from multiplying the derivatives whose variables correspond to the edge $e$ in each $\hat{f}_i$. That is,
$$
\left(\prod_{i\in e}\hat{f}_i\right) \tr(A^{k(k-1)}) = \left( \prod_{\substack{i,j\in e \\  i\neq j}} \frac{\partial}{\partial A_{i,j}}\right) \tr(A^{k(k-1)}).
$$
Evidently, this quantity is a constant which only  depends on the rows and columns of $A$ that are indexed by the elements of $e$, and does not depend upon the particular edge under consideration. So if we let $\bar{A}$ be the $k\times k$ matrix with new variable entries $\bar{A}_{ij}$ reindexed by their rows and columns, and denote the desired constant by $C$, we have that
$$
C =  \left( \prod_{\substack{ i,j\leq k\\i\neq j} } \frac{\partial}{\partial \bar{A}_{i,j}}\right)\tr(\bar{A}^{k(k-1)}).
$$
Thus, we can write
$$
\frac{\Tr_k(H)}{k} = \frac{(k-1)^{n-1}}{k\left((k-1)!\right)^k} C |E(H)|,
$$
and it only remains to determine the value of $C$.

Recall the form of the trace,
$$
\tr(\bar{A}^{k(k-1)}) = \sum_{i_1, i_2, \ldots, i_{k(k-1)}=1}^k \bar{A}_{i_1, i_2} \bar{A}_{i_2, i_3} \ldots \bar{A}_{i_{k(k-1)-1}, i_{k(k-1)}} \bar{A}_{i_{k(k-1)}, i_1},
$$
and observe that a term of this sum survives the differentiation operator only when it consists of a permutation of the $k(k-1)$ non-diagonal elements of $\bar{A}$.  Also note that such a term reduces to $1$ if it survives.  Hence $C$ is the number of trace terms that are orderings of the non-diagonal elements of $\bar{A}.$

To count the number of such terms, consider $D_k$, the complete labeled directed graph on $k$ vertices.  Let the edge $(i,j)$ be labeled by $\bar{A}_{ij}$.  Then the terms of $\tr(\bar{A}^{k(k-1)})$ that we are trying to count are in bijection with Eulerian cycles in $D_k$ where one edge (the first) is distinguished.  Hence each Eulerian cycle in $D_k$ corresponds to exactly $k(k-1)$ surviving trace terms, by choosing which edge to start from. Thus $C = k(k-1)(\# \text{ of Eulerian cycles in } D_k)$.

Counting Eulerian cycles in a directed graph can be done using so-called ``BEST Theorem'' (\cite{st41,be51}), which says that the number of such cycles in a directed graph $D$ is given by
$$
t_w(G)\prod_{v\in D}(\deg^{-}(v) -1)!,
$$
where $\deg^{-}(v)$ is the indegree of vertex $v$ (which must equal the outdegree for $D$ to be Eulerian), and $t_w(G)$ is the number of \emph{arborescences} of $D$ rooted at $w$.  (An \emph{arborescence} of $D$ at $w$ is a spanning tree rooted at $w$, with all edges pointing away from $w$.)  Remarkably, $t_w(G)$ does not depend on the choice of vertex $w$.

In our case, note that since $D_k$ is complete, every possible (undirected) spanning tree can be realized as a subgraph, and the requirement that the edges point away from $w$ gives exactly one possible way of orienting the edges for each such tree. Hence the number of arborescences of $D_k$ is precisely the number of labeled trees on $k$ vertices, which Cayley's Formula says is $k^{k-2}$ \cite{cay89,pru18}.

Hence we see that
$$
C = k(k-1) k^{k-2} (k-2)!^k,
$$
so that
\begin{align*}
\frac{\Tr_k(H)}{k} & =  \frac{(k-1)^{n-1}}{k\left((k-1)!\right)^k} C |E(H)| \\
& = \frac{(k-1)^{n-1}}{k\left((k-1)!\right)^k}  k(k-1) k^{k-2} (k-2)!^k |E(H)|  \\
& = k^{k-2}(k-1)^{n-k} |E(H)|,
\end{align*}
completing the proof.
\end{proof}

We can follow a similar procedure (with a less general calculation of the constant) to determine the next coefficient. Before doing so, we introduce a generalization of the triangle graph, a hypergraph which appears as the subgraph counted by the codegree $k+1$ coefficient.

\begin{definition} A \emph{simplex} in a hypergraph is a set of $k+1$ vertices where every set of $k$ vertices forms an edge. \end{definition}

In a graph, the simplex is a triangle. In a $3$-uniform hypergraph, the simplex is $4$ vertices, each set of three forming an edge. This hypergraph can be visualized as a tetrahedron in $\mathbb{R}^3$, where the facets of the tetrahedron are edges of the hypergraph.

\begin{lemma} \label{simplemma} The simplex is the only $k$-uniform $k$-valent multihypergraph with $k+1$ edges. \end{lemma}

\begin{proof} Suppose $H$ is a $k$-uniform $k$-valent multihypergraph with $k+1$ edges. It is clear that $H$ must have at least $k+1$ vertices. Count pairs of the form (vertex, edge) where the vertex lies on the edge. Since there are $k+1$ edges, each containing exactly $k$ vertices, we count $k(k+1)$ pairs. On the other hand, if there were more than $k+1$ vertices in such a hypergraph, $k$-valency would imply that there were strictly more than $k(k+1)$ such pairs. Hence $H$ has exactly $k+1$ vertices.

Since each edge has $k$ vertices, we can label each edge by the vertex it does not contain. Each vertex must be used in exactly $k$ edges by $k$-valency. Hence for each vertex there is a unique edge not containing that vertex. Every possible $k$-set of vertices forms an edge, and $H$ is a simplex.
\end{proof}

\begin{theorem} \label{codegreeplus1}The codegree $k+1$ coefficient of the characteristic polynomial of a $k$-graph $H$  is $-C (k-1)^{n-k} (\# \text{ of simplices in } H), $ where $C$ is a constant depending only on $k$. \end{theorem}

\begin{proof} The codegree $k+1$ coefficient is $$P_{k+1}\left(-\frac{\Tr_1(H)}{1}, -\frac{\Tr_2(H)}{2}, \ldots , -\frac{\Tr_{k+1}(H)}{{k+1}}\right) $$

By Corollary \ref{zerotraces}, only the last two parameter values are non-zero. Recalling that monomials in $P_{k+1}$ have indices that sum (with multiplicity) to $k+1$, it is easy to see that the only term remaining in this expression is $-\frac{\Tr_{k+1}(H)}{{k+1}}$. Applying the definition of this trace,

$$
\frac{\Tr_{k+1}(H)}{{k+1}} = \frac{(k-1)^{n-1}}{{k+1}} \sum_{d_1+d_2+\ldots+d_n={k+1}} \left( \prod_{i=1}^n\frac{\hat{f}_i^{d_i}}{((k-1)d_i)!} \right) \tr(A^{(k-1)(k+1)})
$$

Next we show that only terms with each $d_i\leq 1$ can possibly contribute to the sum, again by showing that any term of $\mathcal{D} =\ \prod_{i=1}^n \hat{f}_i^{d_i}$ that uses some $\hat{f}_i$ at least twice can yield no surviving terms.

If $\hat{f}_i$ is used twice, then $i$ is used as the first index of a differentiation variable $A_{ij}$ at least $2(k-1)$  times. Therefore, in order for monomials in $\tr(A^{(k-1)(k+1)})$ to survive $\hat{f}_i$, $i$ must be used at least $2(k-1)$ times as a second index. Note that $i$ cannot be used as a second index of differentiation in $\hat{f}_i$, since $i$ and the second indices of differentiation of any term must form an edge of our hypergraph to have a non-zero coefficient in $\hat{f}_i$. As there are only $k-1$ other factors $\hat{f}_j$,  some $\hat{f}_j$ must use $i$ as a secondary index at least twice in a single monomial coefficient. However, as noted above, $j$ and the secondary indices must form an edge to have a nonzero coefficient in $\hat{f}_j$, so $i$ cannot be used twice.  Thus we see that any operator using some $\hat{f}_i$ more than once has no survivors, and the only operators that contribute are those where each exponent is one or zero.  That is,
$$
\frac{\Tr_{k+1}(H)}{{k+1}}  = \frac{(k-1)^{(n-1)}}{(k+1) \left[ (k-1)! \right ]^{(k+1)}} \sum_{i_1<i_2<\cdots<i_{k+1}} \prod_{j=1}^{k+1} \hat{f}_{i_j} \tr(A^{(k-1)(k+1)}).
$$

Fix vertices $i_1, i_2, \ldots i_{k+1}$. By Lemma \ref{simplemma}, the simplex is the only $k$-valent $(k+1)$-edge $k$-graph, so the only survivors of $\prod_{j=1}^{k+1} \hat{f}_{i_j}$ when applied to $\tr(A^{(k+1)(k-1)})$ arise from a simplex on vertices $i_1, i_2, \ldots , i_{k+1}$.
Hence, if we let $S$ be the family of all vertex sets of simplices of $H$,
$$
\frac{\Tr_{k+1}(H)}{{k+1}} = \frac{(k-1)^{(n-1)}}{(k+1)\left[(k-1)!\right]^{(k+1)}} \sum_{s\in S} \prod_{v\in s} \hat{f}_{v} \tr(A^{(k-1)(k+1)}).
$$
Note that $\prod_{v\in s} \hat{f}_{v} \tr(A^{(k-1)(k+1)})$ only depends on the rows and columns appearing in $s$, and so is independent of the choice of vertices and size of $A$. Thus $\prod_{v\in s} \hat{f}_{v} \tr(A^{(k-1)(k+1)}) = C^\prime$ is a constant depending on $k$. So we see that
$$
\frac{\Tr_{k+1}(H)}{{k+1}}  = \frac{(k-1)^{(n-1)}}{(k+1)\left[(k-1)!\right]^{(k+1)}} C^\prime \cdot (\# \text{ of simplices in } H).
$$

Gathering all of our constants, we find that the codegree $k+1$ coefficient of the characteristic polynomial of $H$ is

$$-C(k-1)^{n-k}(\# \text{ of simplices in } H),$$ as desired. \end{proof}

Clearly absent from the proof above is the determination of the constants $C = C_k$.  For graphs, it is well known \cite{big93} that the codegree 3 coefficient is $-2(\# \text{ of triangles in } G)$, i.e., $C_2=2$.

For $k>2$, this constant can be found by computing the codegree $k+1$ coefficient of the characteristic polynomial for the simplex\footnote{See Section 6 for a more detailed description of this and other computations we employ.}, and solving for $C$ in Theorem \ref{codegreeplus1}.  By carrying out such a calculation, we can show that
\begin{align*}
C_2 & = 2\\
C_3 & = 21 \\
C_4 & = 588 \\
C_5 & = 28230.
\end{align*}

One can, in principle, perform similar calculations for any fixed uniformity $k$ and fixed codegree $d$ to determine which $k$-valent multi-subgraphs on $d$ edges are being counted, and in what  multiplicity. In practice, the calculations become unwieldy for even modest values of uniformity and codegree. Instead, we hope that a characterization akin to that of the graph case, where the coefficients count ``sesquivalent'' subgraphs with coefficients based on their rank \cite{big93}, can be found for $k$-graphs.  Such a characterization will almost surely depend on a much better understanding of the symmetric hyperdeterminant.

\section{Spectra of Special Hypergraphs}
\subsection{General $k$-partite Hypergraphs}
A $k$-graph $H$ is called $k$-partite, or a $k$-cylinder, if the vertices of $H$ can be partitioned into $k$ sets so that every edge uses exactly one vertex from each set. The best known case is that of a $2$-cylinder, a.k.a.~a bipartite graph.  There are several proofs of the following characterization of bipartite graphs (q.v.~\cite{big93,cve80}).

\begin{theorem} A graph $G$ is bipartite if and only if its multiset spectrum is symmetric about the origin.\end{theorem}

This theorem can be restated as saying that a graph is bipartite if and only if its (multiset) spectrum is invariant under the action of multiplication by any second root of unity. We generalize this to $k$-cylinders.

\begin{theorem} \label{kpartite} The (multiset) spectrum of a $k$-cylinder is invariant under multiplication by any $k$-th root of unity. \end{theorem}

Of course, this is only one direction of the theorem from the graph case. Unfortunately, the converse is true not for $k > 2$. Let $H$ be the unique $3$-uniform hypergraph on four vertices with three edges, i.e., a tetrahedron with one face removed.  It is easy to see that $H$ is not tripartite, but a calculation of the characteristic polynomial reveals
\begin{align*}
\phi_H(\lambda) = & \lambda^{11}(\lambda^3-12)(\lambda^3-1+2i)^3(\lambda^3-1-2i)^3,
\end{align*}
whose roots are symmetric under multiplication by any third root of unity.

It is worth mentioning that if one weakens this statement of the theorem to concern the spectrum as a set instead of a multiset, it be can proved easily using analytic methods.

\begin{proof} Let $H$ be a $k$-cylinder, and $\phi_H(\lambda)$ be its characteristic polynomial. We note that $\phi_H(\lambda)$ (or more generally, any univariate monic polynomial) has its multiset of roots invariant under multiplication by any $k$-th root of unity if and only if  $\phi_H(\lambda) = \lambda^rf(\lambda^k)$ for some polynomial $f$ and integer $r\geq 0$.  This condition is equivalent to having every non-zero coefficient of $\phi_H(\lambda)$ being a term with codegree $0 \pmod k$.

Suppose that there is a non-zero coefficient for a monomial of $\phi_H(\lambda)$ with codegree $i$. Then there exists a multi-subgraph $H^\prime$ of $H$ with $|E(H^\prime)| = i.$ By Theorem \ref{kvalent}, the monomial describing $H^\prime$ is also $k$-valent. Hence every vertex is used $0 \pmod k$ times. As $H$ is a $k$-cylinder, we can count the edges in $H^\prime$ by counting the vertices used in any single partition class. Since \emph{each} vertex of $H^\prime$ is used $0 \pmod k$ times, the number of vertices of $H^\prime$ in any partition class is also $0 \pmod k.$ Hence $i=|E(H^\prime)| \equiv 0 \pmod k$. \end{proof}

\subsection{One Edge Hypergraphs}

Finding the (set) spectrum of a single edge is simple.  Determining the multiplicities of the eigenvalues is a bit more difficult.  Nonetheless, with the help of Theorems \ref{codegreek} and \ref{kpartite}, we can calculate the characteristic polynomial of a single edge hypergraph for any uniformity.

\begin{theorem} \label{thm:oneedge} If $H$ is the $k$-graph with $k$ vertices and a single edge,
$$
\phi_H(\lambda) = \lambda^{k(k-1)^{k-1}-k^{k-1}}(\lambda^k-1)^{k^{k-2}}.
$$
\end{theorem}

\begin{proof} Suppose $\lambda \neq 0$ is an eigenvalue of $H$.  Let $\bx$ be a corresponding eigenvector. If $\bx$ has a zero entry, then the eigenvalue equation (\ref{eigeq}) for a non-zero entry $x_i$ gives
$$
\lambda x_i^{k-1} = \prod_{j\neq i} x_j  = 0.
$$
This contradicts $\lambda\neq 0, $ so we see $\bx$ has no zero entries.  Multiplying all $k$ of the eigenvalue equations, we see
$$ \lambda^k\prod_{i=1}^kx_i^{k-1} = \prod_{i=1}^k \prod_{j \neq i} x_j = \prod_{i=1}^k x_i^{k-1}.$$ Hence we see that $\lambda^k=1$ for any non-zero eigenvalue $\lambda$.

$H$ is $k$-partite, so Theorem \ref{kpartite} implies that
\begin{equation} \label{ep}
\phi_H(\lambda) =  \lambda^a(\lambda^k-1)^{b}.
\end{equation}
Since the characteristic polynomial has degree $k(k-1)^{k-1}$, we also have that $a+kb = k(k-1)^{k-1}$.  The codegree $k$ coefficient in (\ref{ep}) is $-b$, while Theorem \ref{codegreek} implies that the codegree $k$ coefficient is  $-k^{k-2}$.  The formula for $\phi_H(\lambda)$ follows.
\end{proof}

\subsection{Cartesian Products}

Given two hypergraphs $G$ and $H$, the Cartesian product of $G$ and $H$ is the hypergraph $G \Box H$ with $V(G \Box H) = V(G) \times V(H)$ and
$$
E(G \Box H) = \{\{v\} \times e : v \in V(G), e \in E(H)\} \cup \{e \times \{v\} : e \in E(G), v \in V(H)\}.
$$

The following is a natural hypergraph analogue of a standard result from Spectral Graph Theory.

\begin{theorem} \label{cartprod} If $G$ and $H$ are $k$-graphs, and  $\lambda$ and $\mu$ are eigenvalues for $G$ and $H$ respectively, then $\lambda + \mu$ is an eigenvalue for $G\Box H.$\end{theorem}
\begin{proof}
Let $G$ and $H$ be $k$-graphs on $n$ and $m$ vertices, respectively. Let $(\lambda, \bu)$ be an eigenpair for $G$, and let $(\mu, \bv)$ be an eigenpair for $H$.

Define $\mathbf{w} \in \mathbb{C}^{nm}$ to be a vector with entries indexed by pairs $(a,b) \in [n] \times [m]$ so that $w_{(a,b)} = u_a v_b$.
We claim that $\mathbf{w}$ is an eigenvector of $G \Box H$ with eigenvalue $\lambda + \mu$. To verify this, we simply check the eigenvalue equation (\ref{eigeq}) for an arbitrary vertex $(a,b)$ in $G \Box H$.
\begin{align*}
\sum_{e\in (G\Box H)(a,b)} \hspace{-.25in} w^e & = \sum_{\substack{\{a\}\times e\in (G\Box H)(a,b)\\ \text{with } e\in H(b)}} \hspace{-.25in} w^{\{a\}\times e} +\sum_{\substack{e\times \{b\}\in (G\Box H)(a,b)\\ \text{with } e\in G(a)}} \hspace{-.25in} w^{e\times \{b\}} \\
& = \sum_{e\in H(b)} u_a^{k-1}v^e + \sum_{e\in G(a)} u^e v_b^{k-1} \\
& = u_a^{k-1}\sum_{e\in H(b)} v^e + v_b^{k-1}\sum_{e\in G(a)} u^e  \\
& = u_a^{k-1}\mu v_b^{k-1} + v_b^{k-1} \lambda u_a^{k-1} \\
& = (\lambda+\mu)w_{(a,b)}^{k-1} \end{align*}
Since the vertex chosen was arbitrary, each eigenvalue equation holds, so that $(\lambda+\mu, \mathbf{w})$ is an eigenpair for $G\Box H.$
\end{proof}

Theorem \ref{cartprod} implies that
\begin{equation} \label{speceq}
\spec(G \Box H) \supseteq \spec(G) + \spec(H).
\end{equation}
Since $|\spec(G)| = n (k-1)^{n -1}$ and $|\spec(H)| = m (k-1)^{m -1}$ (as multisets) the multiset-sum consists of
$$
nm (k-1)^{n + m - 2} \leq nm (k-1)^{nm-1} = |\spec(G \Box H)|
$$
eigenvalues. Equality in the above only occurs when $k=2$, so the theorem leaves open the possibility that there exist more eigenvalues of a Cartesian product than just those arising from sums of eigenvalues from the factor hypergraphs.  Unfortunately, the reverse inclusion of (\ref{speceq}) is indeed false for $k > 2$, as shown in the next section.

\subsection{The Ultracube} An important and much-studied sequence of graphs are the hypercubes $Q^d$, i.e., the iterated Cartesian product of a single edge with itself.  We extend this definition to higher uniformity, and make some progress in describing its spectrum.

\begin{definition}
Let $E_k$ be the single-edge $k$-graph.  The $d$-dimension $k$-uniform \emph{ultracube} is defined by $Q^d_k = E_k^{\Box d}$.
\end{definition}

As an example of Theorem \ref{cartprod}, first recall that (as a set) $\spec(E_3) = \{0,1,\zeta_3,\zeta_3^{2} \}$, where $\zeta_3$ is a primitive third root of unity in $\mathbb{C}$.  Then Theorem \ref{cartprod} yields
$$
\{-1,0,1,2,\zeta_3,2\zeta_3,\zeta_3^2,2\zeta_3^2,-\zeta_3,-\zeta_3^2\} \subseteq \spec(Q^2_3).
$$
On the other hand, a computation of the  characteristic polynomial gives that
$$
\phi_{Q_3^2}(\lambda) = (\lambda^3 - 1)^{18}  (\lambda^3 - 2)^{27} (\lambda^3 + 1)^{54} \lambda^{549}  (\lambda^{3} - 2)^{486}.
$$
This reveals three additional eigenvalues $\{\sqrt[3]{2}, \sqrt[3]{2}\zeta_3, \sqrt[3]{2}\zeta_3^{2} \}$ for $Q_3^2$, which shows that the set inclusion in Theorem \ref{cartprod} (i.e., (\ref{speceq})) cannot be turned into an equality. Call the eigenvalues of a cartesian product that are not given by Theorem \ref{cartprod} \emph{sporadic} eigenvalues of the hypergraph.

An eigenvector for the real eigenvalue in this sporadic set of eigenvalues for $Q_3^2$ is illuminating, so we describe it explicitly. Let the vertices of $Q_3^2$ be indexed by pairs $(i,j)$ with $ i,j \in [3]$.   It is easy to verify that the vector given by $$x_{(1,1)}=\sqrt[3]{2},$$ $$ x_{(1,2)}=x_{(1,3)}=x_{(2,1)}=x_{(3,1)}=1,$$ $$x_{(2,2)}=x_{(2,3)}=x_{(3,2)}=x_{(3,2)}= 0 $$
is an eigenvector for $Q_3^2$ corresponding to the eigenvalue $\sqrt[3]{2}.$

The existence of the other two sporadic eigenvalues can be deduced by the following (simple) fact.

\begin{lemma}  If $G$ and $H$ are both $k$-partite $k$-graphs, then so is $G\Box H$. \end{lemma}

\begin{proof} Let $c_1:G \rightarrow \mathbb{Z}_k$ and $c_2:H \rightarrow \mathbb{Z}_k$ induce vertex partitions of $G$ and $H$ respectively.  Define a coloring  $c: G\Box H \rightarrow \mathbb{Z}_k$ by $c((v,w)) = c_1(v)+c_2(w)$; it is a simple matter to check that this is a proper $k$-partition.    \end{proof}

Since $E_3$ is trivially tripartite, the lemma gives that $Q_3^2$ is tripartite as well. Thus Theorem \ref{kpartite} gives us the other two sporadic eigenvalues.

The reason that the eigenvector described above for the sporadic eigenvalue $\sqrt[3]{2}$ of $Q_3^2$ is interesting is that it can be generalized to give sporadic eigenvalues of $Q_k^d$ for any $k>2$ and $d>1$. Let the vertices of $Q_k^d$ be labeled by $(i_1, i_2, \ldots, i_d)$ where $i_j\in [k]$.
Define a vector by  $$ x_{(i_1, i_2, \ldots, i_d)} = \begin{cases} \sqrt[k]{d} & \text{ if } i_1=i_2=\ldots = i_d = 1 \\
1 & \text{ if exactly one } i_j\neq 1 \\
 0 & \text{ otherwise. } \end{cases} $$

It is again easy to verify that this vector is an eigenvector for $\lambda = \sqrt[k]{d}$.  Hence we obtain a recursive way to produce eigenvalues of $Q_k^d$ that gives more eigenvalues than simply applying Theorem (\ref{cartprod}).

\begin{theorem} \label{ultracube} Let $\zeta_k$ be a primitive $k$-th root of unity, and define $S = \{0\} \cup \{\zeta_k^j\}_{j=0}^{k-1}$.  Then
$$
\spec(Q_k^d) \supseteq (\spec(Q_k^{d-1}) + S) \cup \{\sqrt[k]{d}, \zeta_{k}\sqrt[k]{d}, \ldots, \zeta_k^{k-1}\sqrt[k]{d}\}
$$
\end{theorem}

\begin{proof} The first set in the union comes from the facts that $Q_k^d = Q_k^{d-1}\Box E_k$ and $S = \spec(E_k)$ by Theorem \ref{thm:oneedge}. The second set are those described in the  preceding paragraph.  \end{proof}

\subsection{Complete $k$-Cylinders}
In this section, we provide a description of the spectrum of the complete $k$-cylinder for any uniformity $k$ and any partition sizes.

Let $H$ be a $k$-cylinder with $A_1, \ldots, A_k$ as its partition sets, so that for any choice of  $v_1\in A_1, v_2\in A_2, \ldots, v_k\in A_k$, we have $\{v_1, v_2, \ldots, v_k\} \in E(H).$ We call $H$ a complete $k$-cylinder.

The eigenvalue equations $\lambda\mathbf{x}^{[k-1]} = \cA_H\mathbf{x}^{k-1}$ for such a hypergraph have a particularly simple  form. For each vertex $v\in A_i$, the corresponding equation is
\begin{equation} \label{compkparteigeq} \lambda x_v^{k-1} = \prod_{\substack{j\in [k]\\ j\neq i}}\left( \sum_{w\in A_j} x_w\right)\end{equation}

\begin{theorem} \label{compkpart} Let $H$ be a complete $k$-cylinder with parts $A_1, \ldots, A_k$, and let $\zeta_{k-1}$ be a primitive $(k-1)$-st root of unity.  Then $\lambda\neq 0$ is an eigenvalue of $H$ if and only if
\begin{equation} \label{compkparteq}
\lambda^k = \prod_{i=1}^k \left(\sum_{v\in A_i} \zeta_{k-1}^{\ell_v}\right)^{k-1}
\end{equation}
for some choice of integers $\ell_v \in \{0,\ldots,k-2\}$ for each $v \in V(H)$.
\end{theorem}

\begin{proof} To show sufficiency, note that by Theorem \ref{kpartite}, it suffices to prove that one of the $k$ roots of the above equation is an eigenvalue.  Specifically, for $i \in [k]$ let $m_i = \sum_{v\in A_i} \zeta_{k-1}^{\ell_v}$, and for each such $i,$ we fix one of the $k$ values of $m_i^{1/k}.$  Then  $\lambda = \prod_{i=1}^k m_i^{(k-1)/k}$ is one of the solutions to (\ref{compkparteq}).

For  $v\in A_i$, we let $x_v = \zeta_{k-1}^{\ell_v}m_i^{-1/k}$.  We  verify that the vector defined thusly is an eigenvector for $\lambda$ by checking the eigenvalue equations (\ref{compkparteigeq}):
\begin{align*}
\prod_{\substack{j\in [k]\\ j\neq i}}\left( \sum_{w\in A_j} x_w\right)& = \prod_{\substack{j\in [k]\\ j\neq i}}\left( \sum_{w\in A_j} \zeta_{k-1}^{\ell_w}m_j^{-1/k} \right) \\
& = \prod_{\substack{j\in [k]\\ j\neq i}}m_j^{-1/k}\left( \sum_{w\in A_j} \zeta_{k-1}^{\ell_w}\right) \\
& = \prod_{\substack{j\in [k]\\ j\neq i}} m_j^{(k-1)/k} \\
& =  \left(\prod_{j=1}^k m_j^{(k-1)/k}\right)m_i^{(-k+1)/k} \\
& =  \left(\prod_{j=1}^k m_j^{(k-1)/k}\right) \left(\zeta_{k-1}^{\ell_v} m_i^{-1/k}\right)^{k-1} \\
& = \lambda x_v^{k-1}.
\end{align*}
To establish necessity, let $(\lambda, \mathbf{x})$ be an eigenpair for $H$ with $\lambda \neq 0$.  Note that for any two  vertices in the same class $A_i$, the defining eigenvalue equation is the same. Hence we see that $x_v = \zeta_{k-1}^\ell x_w $ for some $0\leq \ell < k-1$ whenever $w,v$ are vertices in the same class. In particular, if $\mathbf{x}$ has any zero coordinate, it is zero on some entire class $A_r$. Then $\sum_{v\in A_r} x_v = 0.$ An eigenvector has to have some non-zero coordinate, say $x_v$, whose vertex must then lie in some class $A_s$ with $s\neq r$.  If we look at the defining eigenvalue equation for $x_v$, we see that
$$
\lambda x_v^{k-1} =   \prod_{\substack{j\in [k]\\ j\neq s}}\left( \sum_{w\in A_j} x_w\right) =0.
$$
Since $x_v^{k-1}$ is non-zero, we conclude that $\lambda = 0,$ a contradiction. Hence any eigenvector for a non-zero eigenvalue of $H$ must have all non-zero entries.

As $\mathbf{x}$ has full support, we can assume without loss of generality that $x_1\in A_1$ with $x_1=1.$ Let $a_1=1,$   and for each partition class $A_i$, $i \neq 1$, choose a vector entry $x_{v_i}$ with $v_i \in A_i$, and define $a_i = x_{v_i}.$ Then for each entry $x_v$ with $v\in  A_i$, we have a unique representation $x_v = a_i \zeta_{k-1}^{\ell_v}$ where $0\leq \ell_v< k-1.$ Now if we let $m_i  = \sum_{v\in A_i} \zeta_{k-1}^{\ell_v}$, we have
$$
\sum_{v\in A_i} x_v = a_i m_i.
$$
Note that if $m_i =0$, our eigenvalue equations would give that any class other than $A_i$ has all corresponding entries in $\mathbf{x}$ equal to zero, which contradicts $\mathbf{x}$ having full support. Hence $m_i \neq 0$. From the eigenvalue equation for $x_1$, we have
\begin{equation} \label{x1eq}
\lambda = \prod_{i=2}^k a_i m_i.
\end{equation}
For a vertex in class $A_j$, the eigenvalue equation is
\begin{equation} \label{xjeq}
\lambda a_j^{k-1} = \prod_{i\neq j} a_i m_i.
\end{equation}
From (\ref{xjeq}), we see that
\begin{align*}
a_j^k &  = \frac{a_j m_j \prod_{i\neq j} a_i m_i}{m_j \lambda} \\
& = \frac{ \prod_{i =1}^k a_i m_i}{m_j \prod_{i=2}^k a_i m_i} \\
& = \frac{m_1}{m_j}.
\end{align*}
If we raise both sides of (\ref{x1eq}) to the $k$-th power, we find
\begin{align*}
\lambda^k & = \prod_{i=2}^k a_i^km_i^k \\
& = \prod_{i=2}^k \frac{m_1}{m_i} m_i^k \\
& = \prod_{i=1}^k m_i^{k-1},
\end{align*}
completing the proof.
\end{proof}

It is worth noting that this argument provides a different proof from the ``standard one'' for the spectrum of a complete bipartite graph, although it does lose any information concerning the multiplicities of eigenvalues. To be precise, if we  let $k=2$, $|A_1| =m$, $|A_2|=n$ (with $m+n>2$), then the only $(k-1)$-st root of unity is $\zeta_{k-1} = 1$, and so our theorem gives that $\lambda^2 = mn$.  Therefore, the (set) spectrum of the complete bipartite graph $K_{m,n}$ is $\{- \sqrt{mn}, 0, \sqrt{mn}\}.$

For $k=3$ the theorem also gives a fairly succinct description of the spectrum of a complete tripartite hypergraph.
\begin{corollary} Let $H$ be the complete 3-cylinder with partition sizes $n_1$, $n_2$, and $n_3$, and for $i \in [3]$, let $S_i = \{n_i-2m \:|\: m\in \mathbb{N} \text{ and }  m <n_i/2\}.$   Then
$$
\spec(H) = \{0\} \cup \{ \zeta_3^j\left( s_1s_2s_3\right)^{2/3} |\: 0\leq j <3 \text{ and } s_i \in S_i\}.
$$
\end{corollary}

\subsection{Complete $k$-Graphs}

The complete $k$-graph on $n$ vertices is an obvious next candidate for which to compute the spectrum.  We obtain a complete characterization of the (set) spectrum in the first unknown case $k=3$.  Unfortunately, our methods do not lead to a complete characterization in cases of uniformity greater than $3$, but they do reveal an interesting connection to the elementary symmetric polynomials.

As in the case of complete $k$-cylinders, the eigenvalue equations for a complete $k$-graph have particularly simple form. For any vertex $v$, the eigenvalue equation is given by

\begin{equation} \label{compeigeq1}
\lambda x_v^{k-1} = \sum_{e\in \binom{[n]\setminus \{v\}}{k-1}} \hspace{-.3cm} x^e.
\end{equation}
Note, however, that for any $r\geq 1$,
$$
\sum_{e\in \binom{[n]\setminus \{v\}}{r}} \hspace{-.3cm} x^e = \sum_{e\in \binom{[n]}{r}} x^e - x_v\hspace{-.3cm}\sum_{e\in \binom{[n]\setminus \{v\}}{r-1}} \hspace{-.3cm}x^e.
$$
Applying this identity repeatedly to the eigenvalue equation above, we obtain
$$
\lambda x_v^{k-1} = \sum_{e\in \binom{[n]}{k-1}}\hspace{-.2cm} x^e -x_v\hspace{-.2cm}\sum_{e\in \binom{[n]}{k-2}} \hspace{-.2cm}x^e+x_v^2\hspace{-.2cm}\sum_{e\in \binom{[n]}{k-3}} \hspace{-.2cm}x^e-\ldots +(-1)^{k-1}x_v^{k-1}.
$$
Notice that $\sum_{e\in \binom{[n]}{r}} x^e$ is the sum over all square-free monomials of degree $r$ in $n$ variables, i.e.,  precisely the degree $r$ elementary symmetric polynomial in $n$ variables. We denote this polynomial by $E_{r}(\mathbf{x})$ (letting $E_0 \equiv 1$) and use it to rewrite the eigenvalue equation more succinctly as

\begin{equation} \label{compeigeq}\lambda x_v^{k-1} = E_{k-1}(\mathbf{x})-x_v E_{k-2}(\mathbf{x})+x_v^2 E_{k-3}(\mathbf{x}) - \ldots +(-1)^{k-1} x_v^{k-1}.  \end{equation}

\begin{theorem} \label{complete3unif} The complete $3$-uniform hypergraph on $n$ vertices has eigenvalues $0, 1, \binom{n-1}{2}$, and at most $2n$ others, which can be found by substituting the roots of one of $n/2$ univariate quartic polynomials into a particular quadratic polynomial. In principle, the roots can be obtained from an explicit list of $O(n)$ formulas. \end{theorem}

By ``explicit'', we mean that it is indeed possible to write down a list of $O(n)$ expressions involving $O(1)$ algebraic operations that give the roots in terms of $n$ and which are each $O(\log n)$ symbols long.  However, doing so yields absurdly long expressions that are neither useful nor enlightening, so we omit them.

\begin{proof} We first claim that any eigenpair $(\lambda, \mathbf{x})$ where $\lambda\notin \{0,1\}$ must have the property that $\mathbf{x}$ has full support. To see this, let $\mathbf{x}$ be an eigenvector without full support; we show $\lambda^2 = \lambda$.

Equation (\ref{compeigeq}) applied to a vertex $w$ so that $x_w = 0$ yields $E_2(\mathbf{x}) = 0$.  Summing equations (\ref{compeigeq1}) over all vertices yields
$$
\lambda\left( \sum_{v=1}^nx_v^2\right) = (n-2)\sum_{e\in \binom{[n]}{2}}x^e= (n-2)E_2(\mathbf{x}) = 0.
$$
Since $\lambda \neq 0$, the sum of the squares of the entries of $\mathbf{x}$ is zero. From this, it follows that
$$
\left(\sum_{v=1}^nx_v\right)^2 = \sum_{v=1}^nx_v^2 + 2\sum_{e\in \binom{[n]}{2}}x^e = 0,
$$
so that $\sum_{v} x_v =E_1(\mathbf{x}) = 0.$ Thus equation (\ref{compeigeq}) reduces to
$$
\lambda x_v^2 = x_v^2
$$
for each $v$, from which the claim follows, since there must be some $v$ so that $x_v \neq 0$.

Now let $(\lambda, \mathbf{x})$ be an eigenpair with $\lambda\notin\{0,1\}$, so that $\mathbf{x}$ has full support. Consider the polynomial $g\in \mathbb{C}[y] $ given by $g(y) = (1-\lambda)y^2-yE_1(\mathbf{x})+E_2(\mathbf{x}).$  Noting that every coordinate of $\mathbf{x}$ is root of $g$, we see that $\mathbf{x}$ has at most two distinct entries. If all entries are the same, a quick calculation shows that the corresponding eigenvalue is $\binom{n-1}{2}.$ So, assume there are exactly two entries. By rescaling the vector, we may assume without loss that at least half of the entries are $1$, and we denote the other entry by $c$. Let $t \leq n/2$ be the number of times $c$ appears in $\mathbf{x}$. Then there are only two eigenvalue equations: one for the entry $1$,
\begin{equation} \label{comp3uniflambda}
\lambda = \binom{t}{2}c^2 + t(n-t-1)c + \binom{n-t-1}{2},
\end{equation}
and another for the entry $c$,
\begin{equation}
\lambda c^2 = \binom{t-1}{2}c^2 + (t-1)(n-t)c + \binom{n-t}{2}.
\end{equation}
Substituting the first into the second yields the quartic polynomial
\begin{align*}
P(c) & = \binom{t}{2}c^4 + t(n-t-1)c^3 + \left(\binom{n-t-1}{2} - \binom{t-1}{2}\right) c^2 \\
& \qquad - (t-1)(n-t)c - \binom{n-t}{2}.
\end{align*}
Then any non-trivial root $c_0$ of $P(c)$ and the value $\lambda_0$ obtained by substituting $c_0$ for $c$ in (\ref{comp3uniflambda}) yields an eigenpair $(\lambda_0, \mathbf{x_0}),$ where $c_0$ appears $t$ times in $\mathbf{x_0}$ and all other entries are 1. As $t$ takes on at most $n/2$  values, and each leads to at most 4 eigenvalues, there can be no more than $2n$ additional eigenvalues.
\end{proof}

\section{Conclusion and Open Problems}

Because Spectral Graph Theory has been such a rich font of interesting mathematics, the list of natural ``next'' questions about hypergraph spectra is virtually endless.  Here we outline a few of those that we find particularly appealing.
\begin{enumerate}
\item What is the spectrum of the complete $k$-graph for $k > 3$?  What are the multiplicities for $k = 3$?
\item What do the spectra of other natural hypergraph classes look like?  For example, one might consider the generalized Erd\H{o}s-R\'{e}nyi random hypergraph, Steiner triple systems, Venn diagrams, etc.
\item Fully describe the eigenvalues of Cartesian products -- in particular, explain the ``sporadic'' ones.  Ultracubes are a natural object of study in this vein.
\item How does one compute the multiplicities of eigenvalues of hypergraphs in general?  Is there a ``geometric multiplicity'' analogous to the dimension of eigenspaces of matrices which provides a lower bound for this ``algebraic'' multiplicity?  Perhaps such an invariant can be defined via the algebraic varieties given by the equations $(\mathcal{A}_H -\lambda \mathcal{I})\bx^{k-1} = \mathbf{0}$, $\lambda \in \spec(H)$, as in the case of graphs.
\item Characterize those hypergraphs whose spectra are invariant under multiplication by $k$-th roots of unity, i.e., find the appropriate weakening of the hypotheses of Theorem \ref{kpartite} to achieve necessity.
\item How can one compute the spectrum of a hypergraph more efficiently?  Our computational experiments have struggled with hypergraphs on as few as six vertices.
\item How does $\spec(H)$ relate to other hypergraph invariants, such as the domination number, transversal number, etc.?
\item Is it true that, for a sequence of $k$-graphs $H$ on $n \rightarrow \infty$ vertices, if the spectrum is ``random-like'', then $H$ is quasirandom in the sense of \cite{chps11} or \cite{krs02}?  Are ``expansion properties'' of hypergraphs related to the size of the second-largest-modulus eigenvalue?
\end{enumerate}


\section{A Note on Computation}

As noted in Section 2.1, the characteristic polynomial of a $k$-graph $H$ is the resultant of the polynomials $F_i = \lambda x_i^{k-1} - \sum_{e\in H(i)} x^e$ where $i\in [n]$.  Hence computing the characteristic polynomial reduces to computing the resultant. For this computation, we use the algorithm described in Chapter 3, Section 4 of \cite{clo98}. We describe the algorithm here for completeness.\\

\noindent $\bullet$ Compute $\res(F_1, F_2, \ldots, F_n)$ as follows:\\

Let $d = n(k-1) -n +1$, and let $S$ be the set of all monomials of degree $d$ in the variables $x_1, \ldots x_n$. (We denote such a
monomial $x^\alpha$, where $x$ stands for a variable vector, and  $\alpha$ stands for an exponent vector.) Let
\begin{align*}
S_1 & = \{ x^\alpha \in S\, |\, x_1^{k-1} \text{ divides } x^\alpha\} \\
S_2 & = \{ x^\alpha \in S\setminus S_1 \,|\, x_2^{k-1} \text{ divides } x^\alpha\} \\
    & \vdots \\
S_n &  = \{ x^\alpha \in S\setminus \bigcup_{i=1}^{n-1} S_i \,|\, x_n^{k-1} \text{ divides } x^\alpha\} \end{align*}

This collection forms a partition of $S$ (by an easy pigeon-hole principle argument).  Fix an ordering on $S$, and define the $|S|\times |S|$ matrix $M$ as follows. The  $(\alpha, \beta)$ entry of $M $ is the coefficient of $x^\beta$ in the polynomial $F_i(x)\frac{x^\alpha}{x_i^{k-1}}$, where $i$ is the unique index such that $x^\alpha \in S_i$. In particular, any non-zero $(\alpha, \beta)$ entry is one of the coefficients of $F_i,$ where $i$ has
 $x^\alpha \in S_i$.

Call a monomial $x^\alpha \in S$  \emph{reduced} if there is exactly one $i$ so that $x_i^{k-1}$ divides $x^\alpha$. Form the matrix $M'$
by deleting the rows and columns of $M$ that correspond to reduced monomials.  The resultant of the system is then $\det(M)/\det(M')$, provided that the denominator does not vanish. In our case, each determinant is actually a characteristic polynomial, so this is never an issue.

Notice that for uniformity $k$, we have $|S| = \binom{n(k-1)+1}{n}$. For $k=3$, $|S| \approx 4^n/\sqrt{n}$. So the matrix $M$ has approximately $16^n/n$ entries. We need the characteristic polynomial of this matrix, a computation which is not obviously parallelizable. Hence the space and time demands of computing the characteristic polynomial of a hypermatrix are high, even for hypergraphs with a small number of vertices.

We implemented this algorithm using the free and open-source mathematics software system Sage, and used it for calculating characteristic polynomials, and then by finding the roots, calculating the spectrum.  The implementation we used, including some of the other routines we wrote to produce the hypermatrices and convert them to and from polynomials, are available at \verb"http://www.math.sc.edu/~cooper/resultants.html".  Running on fairly modest systems (8 core AMD with 8 gigabytes of RAM), we were unable to compute characteristic polynomials for some 3-graphs on only 9 vertices. The results we did obtain, including CPU times, are available on the same website.

\vskip .5cm

\noindent{\bf Acknowledgements.} Thank you to Duncan Buell, Fan Chung, David Cox, Andy Kustin and the anonymous referee for valuable comments and suggestions.

\end{document}